\documentclass[12pt]{amsart}
\usepackage{amsfonts}

\usepackage[margin=1.05in]{geometry}
\usepackage{amsmath,amssymb,amsthm,mathtools}
\usepackage[colorlinks=true,linkcolor=blue,citecolor=blue,urlcolor=blue]{hyperref}
\newtheorem{definition}{Definition}
\newtheorem{conjecture}{Conjecture}
\newtheorem{theorem}{Theorem}

\newtheorem{lemma}{Lemma}

\newtheorem{remark}{Remark}

\newcommand{\ddbar}{i\partial\bar\partial}
\newcommand{\PSH}{\operatorname{PSH}}
\newcommand{\Exc}{\operatorname{Exc}}
\newcommand{\tr}{\operatorname{tr}}

\newcommand{\Amp}{\operatorname{Amp}}

\newcommand{\codim}{\operatorname{codim}}
\newcommand{\Bl}{\operatorname{Bl}}
\newcommand{\NullJ}{\operatorname {Dest}^{opt}}

\date{\today}

\begin{document}
\title{On the Datar--Mete--Song Minimal Slope Conjecture}
\author{Xin Fu}

\begin{abstract}
We prove a conjecture of Datar-Mete-Song \cite{DMS} characterizing
$J$-slope semi-stability by the minimal $J$-slope.  More precisely, for
a semi-stable pair of K\"ahler classes $(\alpha,\beta)$ on a compact K\"ahler manifold $X$, every big and nef birational test
class has slope at least the topological $J$-slope, whereas an unstable
pair admits a test class with strictly smaller slope.  We also  prove that  the set of optimally destabilizing subvarieties of a semi-stable pair $(\alpha,\beta)$  is finite if $X$ is a
compact toric K\"ahler manifold.
In the toric invariant case, we show that Murakami's \cite{Murakami} weak solution to  the $J$-equation is smooth and
K\"ahler on the dense big torus $(\mathbb{C}^*)^n$ of $X$.  
\end{abstract}
\maketitle

\section{Introduction}
The $J$-equation is the critical equation of the $J$-flow introduced by Donaldson \cite{Donaldson}
in the framework of moment maps.
Let $X$ be a compact K\"ahler manifold of complex dimension $n$, and
let $\alpha,\beta$ be K\"ahler classes.  For a K\"ahler form
$\omega\in\beta$, the $J$-equation asks for a K\"ahler form
$\chi\in\alpha$ such that
$$
 n\,\chi^{n-1}\wedge\omega
 =
 \mu(\alpha,\beta)\,\chi^n,
 \qquad
 \mu(\alpha,\beta)
 =
 n\,\frac{\alpha^{n-1}\cdot\beta}{\alpha^n}.
$$
After that both the $J$-equation and the $J$-flow are extensively studied in \cite{Chen1,Chen2,Weinkove1,Weinkove2,FLSW,SongWeinkove,CS,LS,Sze,FL,Chen}. The analytic work of Weinkove and
Song--Weinkove established convergence under suitable pointwise positivity
assumptions \cite{Weinkove1,SongWeinkove}.  In parallel,
Lejmi--Sz\'ekelyhidi \cite{LS} formulated numerical slope conditions and conjectured that the solvability of $J$-equation is equivalent to certain numerical stability condition. Collins-Sz\'ekelyhidi \cite{CS} proved the conjecture in the toric case, and in the general case, Chen \cite{Chen} made a breakthrough by showing that uniform stability implies the existence of a smooth solution to the $J$-equation.   By further extending Chen's work, Datar-Pingali \cite{DatarPingali} and Song \cite{Song} proved Lejmi-Sz\'ekelyhidi conjecture in full generality in the projective case and K\"ahler case separately. There are also many other interesting recent studies related to the $J$-equations, c.f. \cite{CLT,To,FM,GuoSong,SD1,SD2,SD3}.

Now it is well understood that the stable theory is now governed by strict positivity on every
proper analytic subvariety. The recent pioneering work of Datar-Mete-Song \cite{DMS} aim to study the degenerate $J$-equation by going beyond the stable case, where the $J$-equation has a smooth solution,
to the semi-stable case and even the unstable case.  Let us introduce some necessary definitions to further illustrate their expectations.

Let $X$ be a compact K\"ahler manifold of dimension $n$.  Let $\alpha,\beta$ be
K\"ahler classes.  Put
$$
  \mu=\mu(\alpha,\beta):=
  n\,\frac{\alpha^{n-1}\cdot \beta}{\alpha^n}.
$$
For every irreducible analytic subvariety $Z\subset X$ of dimension
$1\le d<n$, put
$$
  \mu_Z=\mu_Z(\alpha,\beta)
  :=d\,\frac{\alpha^{d-1}\cdot\beta\cdot Z}{\alpha^d\cdot Z}.
$$
\begin{definition}[\cite{DMS}]
A pair of K\"ahler class
 $(\alpha,\beta)$ is \textbf{semi-stable} if
for every irreducible analytic subvariety $Z\subset X$ of dimension
$1\le d<n$,
$$
  d\,\alpha^{d-1}\cdot \beta\cdot Z
  \le
  \mu\,\alpha^d\cdot Z.
$$
\end{definition}
Motivated by the surface case, Datar-Mete-Song introduced the so called minimal slope as a numerical invariant to characterize the semi-stability of a pair of K\"ahler classes $(\alpha,\beta)$, which we recall below.

Let  $\pi:Y\to X$ be a birational modification from a compact K\"ahler manifold $Y$ to $X$ and $D$ be an
effective $\mathbb R$-divisor on $Y$, whose cohomology class is
again denoted by $[D]$, such that
$$
  L=\pi^*\alpha-[D]
$$
is a big and nef class, define the slope associated to $(\pi,Y,D)$ as
$$
  \mu_L(\alpha,\beta):=
  n\,\frac{L^{n-1}\cdot\pi^*\beta}{L^n}.
$$
Then the \textbf{minimal slope} is 
\begin{equation}\label{def:minislope}
  \zeta_{\min}(\alpha,\beta):=
  \inf_{\pi,D}\mu_L(\alpha,\beta),
\end{equation}
where the infimum is taken over all such $(\pi,Y, D)$.  Since the test
$(\pi,Y,D)=(\operatorname{id}_X,X,0)$ is admissible, trivially, one has
$$
  \zeta_{\min}(\alpha,\beta)\le \mu.
$$

\begin{conjecture}[Datar--Mete--Song]\label{conj:1}
The semi-stability of the K\"ahler pair $(\alpha,\beta)$ is equivalent to
$$
  \zeta_{\min}(\alpha,\beta)=\mu.
$$
\end{conjecture}
Datar-Mete-Song  confirmed the conjecture in complex dimension two  \cite{DMS} by transforming the J-equation to the complex M\"onge-Ampere equation.
They also proposed the following more ambitious conjecture for arbitrary
pairs, including the unstable case.
\begin{conjecture}[Datar-Mete-Song]\label{conj:weak}
Let $X$ be an $n$-dimensional compact K\"ahler manifold,
and let $\alpha,\beta\in H^{1,1}(X,\mathbb{R})$ be two
K\"ahler classes. For any K\"ahler form $\omega\in\beta$,
there exists a unique K\"ahler current $T\in\alpha$  satisfying
the  $J$-equation
$$
    n\left\langle T^{n-1}\wedge\omega\right\rangle
    =
    \zeta_{\min}\left\langle T^n\right\rangle,
$$
where $\zeta_{\min}=\zeta_{\min}(\alpha,\beta)$ is the
minimal $J$-slope and $\langle\,\cdot\,\rangle$ denotes
the non-pluripolar product.
\end{conjecture}

The main result of this paper is to confirm Conjecture~\ref{conj:1}.
\begin{theorem}
\label{thm:semistable}
Let $(\alpha,\beta)$ be a pair of K\"ahler classes on a K\"ahler manifold of complex dimension $n$.
\begin{itemize}\item[(1)] If  $(\alpha,\beta)$ is semi-stable, then for every birational modification
$\pi:Y\to X$ from a compact K\"ahler manifold $Y$ and every effective
$\mathbb R$-divisor $D$ such that
$L=\pi^*\alpha-[D]$ is big and nef, one has
$$
  n\,L^{n-1}\cdot\pi^*\beta
  \ge
  \mu\,L^n.
$$ Therefore,
$$
  \zeta_{\min}(\alpha,\beta)=\mu.
$$

\item[(2)] If $(\alpha,\beta)$ is not semi-stable, then
$$
  \zeta_{\min}(\alpha,\beta)<\mu.
$$
\end{itemize}
\end{theorem}
\begin{remark}
Item ($2$) of the above theorem was proved by Datar-Mete-Song when $X$ is a projective manifold and both $(\alpha,\beta)$ are polarized K\"ahler classes.\end{remark} In terms of existence of weak solution of $J$-equation in the semi-stable case, 
Murakami's
recent paper \cite{Murakami} confirms the conjecture ~\ref{conj:weak} of Datar-Mete-Song
and the more challenging unstable case remains open. 
\begin{theorem}[Murakami \cite{Murakami}]\label{thm:Rei}
Let $(\alpha,\beta)$ be a semi-stable pair of K\"ahler classes and also let  $\chi\in\alpha$ and 
$\omega\in\beta$ be K\"ahler forms,  then there exists a  closed positive current
$T\in\alpha$ satisfies
$$
 n\left\langle T^{n-1}\wedge\omega\right\rangle
 =
 \mu\left\langle T^n\right\rangle.
$$
Here $\langle\cdot\rangle$ denotes the non-pluripolar product.
\end{theorem}

As a continuation of Murakami's work, in the semi-stable case, it is natural to ask if partial
higher-order regularity of $T$ holds.  To locate
the possible set of degeneracy, it is natural to introduce the following set of optimally destabilizing subvarieties due to Khalid-Sj\"ostr\"om-Dyrefelt.

\begin{definition}\cite{SD2}
    Let  ($\alpha,\beta$) be a semi-stable pair of K\"ahler classes. The optimally destablizing set of subvarieties is defined by
\begin{equation*}
  \NullJ(\alpha,\beta)
  :=
  \bigcup_{\substack{Z\subsetneq X\ \mathrm{irreducible}\\
          1\leq m=\dim Z\leq n-1\\
          \mu_Z=\mu}} Z.
\end{equation*}
Equivalently, $Z$ is contained in $\NullJ(\alpha,\beta)$ precisely
when
\begin{equation*}
  \bigl(\mu\alpha^m-m\alpha^{m-1}\beta\bigr)\cdot Z=0.
\end{equation*}
\end{definition}
To further explore whether partial high-order regularity holds, it seems that a positive answer of the following question of Khalid-Sj\"ostr\"om-Dyrefelt is the first step.

\textbf{Question:}
Is $\NullJ(\alpha,\beta)$ an analytic subset of $X$.

In complex dimension two, by the standard $J$-equation, complex M\"onge-Ampere equation correspondence observed in
\cite{FLSW}, one may  identify $\NullJ(\alpha,\beta)$
with the null locus of a nef and big class and also prove the partial high-order regularity holds out $\NullJ(\alpha,\beta)$. In higher dimension, \cite{SD1}  made important progress on proving that the set of optimal destablizing subvarieties $\NullJ(\alpha,\beta)$ is a finite union of analytic subvarieties of $X$. Our second theorem proves a analogous result  when $X$ is a toric manifold.

 \begin{theorem}\label{thm:toric}Let $X$ be toric K\"ahler manifold of complex dimension $n$ and ($\alpha,\beta$), not necessarily toric invariant, be a semi-stable pair of K\"ahler classes, then \begin{enumerate}
 \item $\NullJ(\alpha,\beta)$ is an analytic subset of $X$.  \item If moreover $(\alpha,\beta)$ are toric invariant classes, then there is a weak solution $T$ of the $J$-equation, which is smooth on the big torus $(\mathbb{C}^*)^n$ of $X$.
 \end{enumerate}
 \end{theorem}
Unfortunately, we are unable to prove high-order regularity of $T$ outside 
$\NullJ(\alpha,\beta)$ even in the toric invariant case.



\medskip

\textbf{Acknowledgements:} The author would like to thank Ved Datar for useful comments, Zakarias Sj\"ostr\"om Dyrefelt for pointing out that the $J$-ll locus in the previous version is essentially defined in his previous works as the set of optimally destablizing subvarieties  and Jian Song for many inspiring lectures on $J$-equations The author also thanks Chatgpt for bringing the reference \cite{McCleerey} to the author's attention and improving the exposition of the paper. The author is supported by National Key R\&D Program of China 2024YFA1014800 and NSFC No. 12401073.
 
\section{Proof of Theorem 1}In this section, we prove Theorem ~\ref{thm:semistable}. Before that, we sketch the proof of item ($1$) of Theorem ~\ref{thm:semistable}.
\begin{enumerate}
\item Step 1: Fix a normalized semi-stable pair $(\alpha,\beta)$ with $\alpha^{n-1}\cdot\beta=\alpha^{n}$. Then for any $0<\epsilon<1$, $(\alpha, \beta_\epsilon:=(1-\epsilon)\beta+\epsilon\alpha)$ is stable. Then by the Nakai-Moishezon criteria \cite{Song}, for every K\"ahler form
$\chi_\varepsilon\in\beta_\varepsilon$, there is a K\"ahler form
$\omega_\varepsilon\in\alpha$ such that
$$
  \chi_\varepsilon\wedge\omega_\varepsilon^{n-1}
  =
  \omega_\varepsilon^n,
  \qquad
  \tr_{\omega_\varepsilon}\chi_\varepsilon=n.
$$
\item Step 2: Fix a admissible test data $(\pi:Y\rightarrow X,D,L)$ as in the definition of minimal slope ~\ref{def:minislope}. The novelty of our approach is  one may construct a current $T$ with Divisorial singularity along $D$, through non pluripolar product  theory \cite{BEGZ}, in class $\pi^*\alpha$ such that
$$
  \int_Y\langle T^n\rangle=L^n\,\,\,\textnormal{and}\,\,\,
  \int_Y\pi^*\beta_\epsilon\wedge\langle T^{n-1}\rangle
  =
  [\pi^*\beta_\epsilon]\cdot L^{n-1}.
$$
Moreover, the potential of the current $T$ is a relative envelop function of a big and nef class.
\item Step 3: Use deep $C^{1,1,}$ regularity of relative envelop function \cite{JWN,McCleerey} and its support property, we prove the following key inequality 
$$ \int_Y\pi^*\beta_\epsilon\wedge\langle T^{n-1}\rangle
  \geq
  \int_Y\langle T^n\rangle.$$
\end{enumerate}
Combining steps $(2)$ and $(3)$, we conclude the proof by letting $\epsilon\rightarrow 0$.

\subsection{Semi-stability implies \texorpdfstring{$\zeta_{\min}(\alpha,\beta)= \mu(\alpha,\beta)$}{1}}
We record  he Nakai-Moishezon criteria for $J$- equation  on a compact K\"ahler manifold. The breakthrough was made by Chen \cite{Chen} and further extended to the projective case by Datar-Pingali \cite{DatarPingali} and to the K\"ahler case by Song \cite{Song}.

\begin{theorem}[\cite{Song}]
\label{thm:song}
Let $\alpha,\beta$ be K\"ahler classes on a compact K\"ahler manifold of dimension $n$, with
$$
  \alpha^{n-1}\cdot\beta=\alpha^n.
$$
Assume that for every irreducible analytic subvariety $Z\subset X$ of dimension
$1\le d<n$,
$$
  d\,\alpha^{d-1}\cdot\beta\cdot Z
  <
  n\,\alpha^d\cdot Z.
$$
Then for every K\"ahler form $\chi\in\beta$, there exists a K\"ahler form
$\omega\in\alpha$ such that
$$
  \chi\wedge\omega^{n-1}=\omega^n.
$$
\end{theorem}
For later purpose, we also introduce the ample locus of a big and nef class.
\begin{definition}
Let $X$ be a compact K\"ahler manifold and 
$L\in H^{1,1}(X,\mathbb{R})$ be a big and nef class.
The \emph{ample locus} of $L$ is defined by
$$
\operatorname{Amp}(L)
:=
\left\{
x\in X \;\middle|\;
\begin{array}{l}
\text{there exists a K\"ahler current } T\in L
\text{ with analytic singularities}\\
\text{such that $T$ is smooth in a neighborhood of $x$}
\end{array}
\right\}.
$$
Set the non-K\"ahler locus of $L$ as the complement of ample locus
$$E_{\mathrm{nK}}(L)
=X\setminus \operatorname{Amp}(L).
$$
\end{definition}
There is an equivalent definition of non-K\"ahler locus due to the deep work of \cite{CollinsTosatti}, which we omit here. It is clear that $E_{\mathrm{nK}}(L)$ is a complex analytic subset of $X$.
The following lemma is well-known.
\begin{lemma}
\label{lem:VL-bounded}
Let $L$ be a big and nef class on a compact K\"ahler manifold,
$\theta_L\in L$ be a smooth representative, and 
$$
  V_L:=\sup\{v\in\PSH(Y,\theta_L):v\le0\}
$$
be the potential with minimal singularities.  Then $V_L$ is locally bounded
on the K\"ahler locus of $L$.
\end{lemma}

We do some preparations for the proof of Theorem ~\ref{thm:semistable}.
The following lemma shows that to compute $\zeta_{\min}(\alpha,\beta)$, one may assume that the effective divisor $D$ in an admissible test data $(\pi,Y,D)$ has rational coefficients. This will be useful in Lemma ~\ref{lem:B}.

\begin{lemma}\label{lem:dq}

Let $\pi:Y\to X$ be a modification from a compact K\"ahler manifold and let
$D$ be an effective $\mathbb R$-divisor such that
$$
  L=\pi^*\alpha-[D]
$$
is big and nef.  Then there exist a further modification
$p:\widetilde Y\to Y$, with $q:=\pi\circ p$, and a sequence of effective
$\mathbb Q$-divisors $D_j$ on $\widetilde Y, j\in \mathbb N^+$ such that
$$
  L_j:=q^*\alpha-[D_j]
$$
is K\"ahler for every $j$ and as $j\rightarrow\infty$
$$
  [D_j]\longrightarrow p^*[D],\qquad
  L_j\longrightarrow p^*L
$$
in $H^{1,1}(\widetilde Y,\mathbb R)$.  In particular,
$$
  n\,\frac{L_j^{n-1}\cdot q^*\beta}{L_j^n}
  \longrightarrow
  n\,\frac{L^{n-1}\cdot\pi^*\beta}{L^n}.
$$
\end{lemma}

\begin{proof}
Pick a higher modification $p:\widetilde Y\to Y$ and set $q=\pi\circ p:\widetilde Y\to X$.  Then there is an effective $q$-exceptional $\mathbb Q$-divisor $E$ such that
$$
  \gamma:=q^*\alpha-[E]
$$
is a K\"ahler class \cite{Varouchas} on $\widetilde Y$.  Put
$$
  \widetilde D:=p^*D,\qquad \widetilde L:=p^*L.
$$
For $0<t<1$, set
$$
  D_t:=(1-t)\widetilde D+tE,
  \qquad
  L_t:=q^*\alpha-[D_t]=(1-t)\widetilde L+t\gamma.
$$
Since $\widetilde L$ is nef and $\gamma$ is K\"ahler, $L_t$ is K\"ahler for every
$t>0$.

Choose a sequence of positive number $t_j\rightarrow 0$. By the openness of  K\"ahler cone and the density of rational numbers, one may assume that the coefficients of $D_{t_j}$ can be perturbed to nonnegative rational numbers and the perturbation is controlled by $t_j$. Thus, one may assume that $D_{t_j}$ already have rational coefficients.

The asserted convergence of slopes follows from continuity of intersection
numbers and the projection formula
$$
  (p^*L)^n=L^n,
  \qquad
  (p^*L)^{n-1}\cdot q^*\beta
  =L^{n-1}\cdot\pi^*\beta.
$$
\end{proof}

In the following lemma, using non-pluripolar product, we construct a current $T$ with divisorial singularity, which is useful to connect class $\pi^*\alpha$ and $\pi^*\alpha-[D]$ in a test data.
\begin{lemma}
\label{lemma:A}
Let $Y$ be a $n$ dimensional compact K\"ahler manifold, also let
\begin{enumerate}
\item  $D$ be an effective $\mathbb R$-divisor, $h_D$ be a smooth metric on the $\mathbb R$ line bundle $D$, $\theta_D$ be the curvature of $h_D$, $s_D$ be the defining section of $D$ and $\varphi_D\:=\log\|s_D\|^2$ be a
divisorial potential satisfying
$$
  \theta_D+\ddbar\varphi_D=[D],
$$
\item $\theta$ be a smooth closed
semipositive $(1,1)$-form with $\theta^n>0$,
\item  $
  L=[\theta]-D
$ be a big and nef class,
$\theta_L:=\theta-\theta_D\in L
$,  \begin{equation}\label{envelop}r=\sup\{s\in\PSH(Y,\theta_L):\varphi_D+s\le0\}^{*}\end{equation}
be the envelope function  and $$T:=\theta+\ddbar(\varphi_D+r)$$ be the associated current in class $[\theta]$.
\end{enumerate}
 
Then, for every smooth closed semipositive $(1,1)$-form $\Theta$, one has
$$
  \int_Y\langle T^n\rangle=L^n\,\,\,\textnormal{and}\,\,\,
  \int_Y\Theta\wedge\langle T^{n-1}\rangle
  =
  [\Theta]\cdot L^{n-1}.
$$
\end{lemma}

\begin{proof}
We first justify that the envelope $r$ is indeed $\theta_L$-psh.  Put
$$
  \mathcal F
  :=
  \{s\in\PSH(Y,\theta_L):\varphi_D+s\le0\}.
$$
It is clear that the family $\mathcal F$ is nonempty.

 Fix a ball $W\Subset Y\setminus\operatorname{Supp}D$, then there is a
constant $C_W$ such that every $s\in\mathcal F$ satisfies
$$
  s\le-\varphi_D\le C_W
  \qquad\text{on }W.
$$ Then by \cite{Demailly-CAG}, $r$ is a $\theta_L$ psh function on $W$. We claim next that $\mathcal F$ is uniformly bounded above on $Y$.
If not, then  there are
$s_j\in\mathcal F$ with $M_j:=\sup_Ys_j\to+\infty$.  Set
$u_j=s_j-M_j$.  Then
$$
  u_j\in\PSH(Y,\theta_L),
  \qquad
  \sup_Yu_j=0.
$$

By  compactness, after passing to a
subsequence $u_j\to u$ in $L^1(Y)$,
$u\not\equiv-\infty$
\cite{Demailly-CAG}.  But on $W$,
$$
  u_j\le C_W-M_j\to-\infty
$$
uniformly, this contradicts the $L^1$ integrability. The uniform upper bound shows that $r$ is $\theta_L$ psh on $Y$ by Hartogs extension Theorem.

For $C\gg1$, the potential $V_L-C$ is admissible in the definition of
$r$, since $\varphi_D+V_L-C\le0$.  Hence
$$
  r\ge V_L-C.
$$
Conversely, since $V_L$ has minimal singularities in the class $L$, every
$\theta_L$-psh potential is bounded above by $V_L+O(1)$.  Thus
$$
  r\le V_L+C'
$$
for some constant $C'$.  Therefore $r$ and $V_L$ have the same
singularity type, and
$$
  R:=\theta_L+\ddbar r
$$
has minimal singularities in the big and nef class $L$.  Moreover
$$
  T
  =
  \theta+\ddbar(\varphi_D+r)
  =
  (\theta_D+\ddbar\varphi_D)+(\theta_L+\ddbar r)
  =
  [D]+R.
$$

Non-pluripolar products are local on the complement of pluripolar sets and
discard the divisorial factor $[D]$ \cite{BEGZ}. On $Y\setminus\operatorname{Supp}D$,
$$
  \langle T^k\rangle=\langle R^k\rangle
  \qquad 1\le k\le n
$$
and both sides extend trivially across $\operatorname{Supp}D$. The non-pluripolar product is closed \cite[Theorem~ 1.8]{BEGZ}.  Since
$R$ has minimal singularities in the big class $L$, its product
$\langle R^k\rangle$ represents the positive product
$\langle L^k\rangle$ in the sense of \cite[Definition~ 1.17]{BEGZ}.  Because
$L$ is nef, this positive product is the ordinary cup product
\cite[paragraph after Definition 1.17]{BEGZ}.  Thus,  for every $1\le k\le n$,
$$
  \{\langle R^k\rangle\}=L^k.
$$
This gives
$$
  \int_Y\langle T^n\rangle
  =
  \int_Y\langle R^n\rangle
  =
  L^n
$$
and
$$
  \int_Y\Theta\wedge\langle T^{n-1}\rangle
  =
  \int_Y\Theta\wedge\langle R^{n-1}\rangle
  =
  [\Theta]\cdot L^{n-1}.
$$
\end{proof}

In the following lemma, we further explore the properties of the current $T$ constructed above. The key is to use the regularity of the envelop function due to \cite{McCleerey}, which extends many previous works \cite{CZ,Tosatti,BermanD,JWN}. 
\begin{lemma}
\label{lem:B}
 Let $(Y,D,\pi)$ be a admissible test data for $(\alpha,\beta) $ with further assumption that $D$ has rational coefficients.   Let $\varphi_D,r,L,\theta$ be defined as in
Lemma~\ref{lemma:A} and $\Theta$ be a smooth
closed semipositive $(1,1)$-forms on $Y$.   set
$$
  u:=r+\varphi_D,\qquad T=\theta+\ddbar u.
$$
Suppose there is an analytic subset
$E\subset Y$, containing
$$
  \operatorname{Supp}D\cup\bigl(Y\setminus\Amp(L)\bigr),
$$
such that $\theta$ is K\"ahler on $Y\setminus E$,  and
$$\label{eqn:J}
  \tr_\theta\Theta=n$$

on $Y\setminus E$.  Then
$$
  \int_Y\Theta\wedge\langle T^{n-1}\rangle
  \ge
  \int_Y\langle T^n\rangle.
$$
\end{lemma}

\begin{proof}
 Firstly, we introduce the following standard    function
\begin{equation}\label{envelopu}
  u=\mathcal P_{\theta}[\varphi_D](0)
  :=
  \sup\{v\in\PSH(Y,\theta):v\le0,\ v\le\varphi_D+C_v\}^{*},
\end{equation}
and we claim that $u=\varphi_D+r.$

On one hand, if $s$ is an admissible function in the definition of $r$ (c.f. \eqref{envelop}), then $$\theta+\ddbar(s+\varphi_D)=(\theta_L+\ddbar s)+(\theta_D+\ddbar\varphi_D)=R+[D]\geq [D],$$ where $R$ is the positive current defined in Lemma~\ref{lemma:A}. Since $s$ is bounded above,  $s+\varphi_D$ is an admissible function in \eqref{envelopu}. Thus $s+\varphi_D\leq u$ and $r+\varphi_D\leq u$.
On the other hand, if
$v\le\varphi_D+C_v$, then $\theta+\ddbar v$ has divisorial part at least $D$.
By Siu decomposition $\theta+\ddbar v-[D]\ge0$.  Direct calculation yields that, on
$Y\setminus D$,
$$\theta_L+\ddbar(v-\varphi_D)=\theta-\theta_D+\ddbar(v-\varphi_D)=\theta+\ddbar v-[D]\geq 0.$$
Hence by Hartogs theorem,
$s=v-\varphi_D$ can be extended as an $\theta_L$-psh function across $D$ to $Y$. On $Y\setminus D$, one has
$\varphi_D+s=v\le0$ and $v$ is upper semi-continuous on $Y$, so $\varphi_D+s=v\le0$ on $Y$. Then $s\leq r$ as an admissible function. So $u\leq \varphi_D+r$.

We remark that  $D$ is a
$\mathbb Q$-divisor, after multiplying by a positive integer $m$ the
divisor $mD$ is integral, and
$$
  \varphi_D=\frac1m\log |s_{mD}|^2_{h}+O(1)
$$
has analytic singularities. Then
\cite[Theorem~1.1 and formula
~(1.4)]{McCleerey} of McCleerey applies directly to
$$
  \theta_{\mathrm{McCleerey}}=\theta,\qquad \gamma=\theta_D,\qquad \psi=\varphi_D,\qquad
  \text{obstacle }u=0.
$$
Indeed, $\varphi_D$ has analytic singularities, $[\theta]$ is big, $[\theta_D]=D$ is
pseudoeffective, and $[\theta]-[\theta_D]=L$ is a big and nef class.   Consequently $u$ has
locally bounded real Hessian on
$\Amp(L)\setminus\operatorname{Supp}D$. Set
$$
  U:=Y\setminus E.
$$
Since $E$ contains $\operatorname{Supp}D\cup(Y\setminus\Amp(L))$, we have
$$
  U\subset \Amp(L)\setminus\operatorname{Supp}D.
$$
Moreover, by McCleerey's formula (1.4), on $U$ one has
$$
  \langle T^n\rangle|_U
  =
  \mathbf{1}_{\{u=0\}}\,\theta^n|_U.
$$
In particular, $\langle T^n\rangle|_U$ charges no mass on 
$U\setminus\{u=0\}$.

 Since $u$ has locally bounded real Hessian on $U$, the current
$T|_U$ is represented by an $L^\infty_{\mathrm{loc}}$ $(1,1)$-form, which 
we denote  by $T_{\mathrm{ac}}$.  Thus $T=T_{\mathrm{ac}}$ as
currents on $U$, and the mixed Bedford--Taylor products on $U$ are
represented by the corresponding almost-everywhere coefficients of
$T_{\mathrm{ac}}$.

Let $q$ be a point where $u$ is twice differentiable. If $q$ further lies in the contact set $\{u=0\}$, then one has $\ddbar u(q)\le0$, because $u\le0$ and equality holds
on the contact set.  Therefore
$$0\le T_{\mathrm{ac}}(q)\le \theta(q).
$$

At such a point $q$, it follows from the assumption \eqref{eqn:J} that 
the following pointwise inequality holds:
$$
  \Theta\wedge T_{\mathrm{ac}}^{n-1}\ge T_{\mathrm{ac}}^n
$$
Indeed, one may diagonalize $T_{ac}$  with respect to $\theta$ and set
$$
  \theta=i\sum_{j=1}^n \,dz_j\wedge d\bar z_j,\quad
  \Theta=i\sum_{j=1}^n  a_{j,\bar k} \,dz_j\wedge d\bar z_k,\quad
  T_{\mathrm{ac}}=i\sum_{j=1}^nt_j \,dz_j\wedge d\bar z_j.
$$
Then $\sum_j a_{j,\bar j}=n$ and $0\le t_j\le1$.  The desired pointwise inequality
$$
  \Theta\wedge T_{\mathrm{ac}}^{n-1}\ge T_{\mathrm{ac}}^n
$$
is equivalent, after dividing ($n-1$)!, to
$$
  \sum_{j=1}^na_{j,\bar j}\prod_{i\ne j}t_i
  \ge
  n\prod_{i=1}^nt_i.
$$
Since $t_j\le1$, one has $\prod_{i\ne j}t_i\geq
\prod_i t_i$.  The inequality follows from $\sum_ja_{j,\bar j}=n$.

Thus on the contact set  $\{u=0\}$, in the Lebesgue-a.e. sense,
$$
  \Theta\wedge T_{\mathrm{ac}}^{n-1}\ge T_{\mathrm{ac}}^n.
$$
On $U\setminus\{u=0\}$, the measure $\langle T^n\rangle$ vanishes, while
$\Theta\wedge\langle T^{n-1}\rangle$ is positive. Finally since the
non-pluripolar products do not charge $E=Y\setminus U$, integration over $Y$ proves
the lemma.
\end{proof}

\medskip
\begin{proof}[\textbf{Proof of item (1) of Theorem~\ref{thm:semistable}}]

Assume $(\alpha,\beta)$ is a pair of semi-stable K\"ahler forms. Without losing of generality, one may assume that $\mu=n$, i.e. 
$$
  \alpha^{n-1}\cdot\beta=\alpha^n.
$$

For $0<\varepsilon<1$, set
$$
  \beta_\varepsilon=(1-\varepsilon)\beta+\varepsilon \alpha.
$$
Then
$$
  \alpha^{n-1}\cdot\beta_\varepsilon=\alpha^n,
$$
and for every irreducible analytic $d$-dimensional subvariety $Z\subset X$, $1\le d<n$,
semi-stability implies that
$$
  d\,\alpha^{d-1}\cdot\beta\cdot Z\le n\,\alpha^d\cdot Z.
$$
It is straightforward to check that ($\alpha,\beta_\varepsilon$) is stable by using the following identity
$$
\begin{aligned}
  n\alpha^d\cdot Z-d\alpha^{d-1}\cdot\beta_\varepsilon\cdot Z
  &=
  (1-\varepsilon)
  \bigl(n\alpha^d\cdot Z-d\alpha^{d-1}\cdot\beta\cdot Z\bigr)\\
  &\quad+\varepsilon(n-d)\alpha^d\cdot Z
  >0.
\end{aligned}
$$
By Theorem~\ref{thm:song}, for every K\"ahler form
$\chi_\varepsilon\in\beta_\varepsilon$, there is a K\"ahler form
$\omega_\varepsilon\in\alpha$ such that
$$
  \chi_\varepsilon\wedge\omega_\varepsilon^{n-1}
  =
  \omega_\varepsilon^n,
  \qquad
  \tr_{\omega_\varepsilon}\chi_\varepsilon=n.
$$
Now fix a test data ($\pi,Y,D$), the pulled back
forms
$$
  \theta_\varepsilon=\pi^*\omega_\varepsilon,\qquad
  \Theta_\varepsilon=\pi^*\chi_\varepsilon.
$$
 are smooth semipositive forms on
$Y$. Moreover outside $\Exc(\pi)$, $\theta_\varepsilon$ is K\"ahler and
$$
  \tr_{\theta_\varepsilon}\Theta_\varepsilon=n.
$$

Apply Lemma~\ref{lemma:A} to the class
$[\theta_\varepsilon]=\pi^*\alpha$ and the divisor $D$, one has a positive current
$$
  T_\varepsilon=\theta_\varepsilon+\ddbar u_\varepsilon
,$$

where $u_\varepsilon:=\sup\{v\in\PSH(Y,\theta_\varepsilon):v\le0,\ v\le\varphi_D+C_v\}^{*},$
such that
$$
  \int_Y\langle T_\varepsilon^n\rangle=L^n,
\qquad
  \int_Y\Theta_\varepsilon\wedge
  \langle T_\varepsilon^{n-1}\rangle
  =
  \pi^*\beta_\varepsilon\cdot L^{n-1}.
$$

Let
$$
  E=\Exc(\pi)\cup\operatorname{Supp}D\cup E_{\mathrm{nK}}(L),
  \qquad E_{\mathrm{nK}}(L):=Y\setminus\Amp(L),
$$
where $\Amp(L)$ denotes the ample locus of the big and nef class $L$.
By lemma~\ref{lem:dq}, one may assume that the divisor $D$ has rational coefficient.
Then, on $Y\setminus E$, the assumption of Lemma~\ref{lem:B} hold.
Hence
$$
  \pi^*\beta_\varepsilon\cdot L^{n-1}
  =
  \int_Y\Theta_\varepsilon\wedge
  \langle T_\varepsilon^{n-1}\rangle
  \ge
  \int_Y\langle T_\varepsilon^n\rangle
  =
  L^n.
$$
Letting $\varepsilon\to0$, one has
$$
  \pi^*\beta\cdot L^{n-1}\ge L^n.
$$

Thus $$
  \zeta_{\min}(\alpha,\beta)\geq\mu.
$$  It is also trivial that 
$$
  \zeta_{\min}(\alpha,\beta)\leq\mu.
$$
This proves item ($1$) of Theorem~\ref{thm:semistable}.
\end{proof}
\subsection{\texorpdfstring{$\zeta_{\min}(\alpha,\beta)< \mu(\alpha,\beta)$}{1} implies unstability}

Now we proceed to prove item ($2$) of Theorem~\ref{thm:semistable}. Therefore for the K\"ahler pair ($\alpha,\beta$), there is a destablizing analytic subvariety $Z\subset X$. We shall fix a test data associated to $Z$ as follows:

\textbf{Set-up:} Let $X$ be a compact complex manifold, let $Z\subset X$ be an
irreducible analytic subvariety of dimension $m$, and set
$$
  c=\codim_X Z=n-m.
$$
Assume $c\ge2$.  Let
$$
  p:\widehat X=\Bl_{\mathcal I_Z}X\to X
$$
be the  blow-up of the reduced ideal sheaf of $Z$, let $E$ be
the total exceptional Cartier divisor such that
$\mathcal I_Z\mathcal O_{\widehat X}=\mathcal O_{\widehat X}(-E)$. Note that over the Zariski open set 
$Z_{\rm reg}$, $p$ is a $\mathbb{P}^{c-1}$ bundle.

\noindent Also let
$\rho:Y\to\widehat X$ be a log resolution of $(\widehat X,E)$ such that $\rho$ is
an isomorphism over 
$p^{-1}(Z_{\rm reg})$, and set
$$
  \pi=p\circ\rho,\qquad F=\rho^*E.
$$
\begin{lemma}
\label{lem:des} Undet the set-up above. If $\theta_F$ is any smooth closed real $(1,1)$-form representing the
first Chern class $c_1(\mathcal O_Y(F))$, then for every smooth closed
$(n-j,n-j)$-form $\eta$ on $X$, one has
\begin{itemize}
\item For $1\le j<c$,
$$
  \int_Y \theta_F^j\wedge\pi^*\eta=0.
$$
\item For $j=c$,
$$
  \int_Y \theta_F^c\wedge\pi^*\eta
  =
  (-1)^{c-1}\int_Z\eta .
$$

\end{itemize}

\end{lemma}

\begin{proof}
Write the effective Cartier divisor $F$ as
$$
  F=\sum_i b_iF_i,\qquad |F|=F_{\rm red}=\bigcup_i F_i,
$$
where the $F_i$ are the irreducible reduced components in the support of
$F$, and $b_i$ is the vanishing order of  $F$
along $F_i$.
Also let 
$$
  [F]=\sum_i b_i[F_i],
$$
be the current of integration associated to the divisor $F$.

Let $L=\mathcal O_Y(F)$ be the associated line bundle and let $s_F\in H^0(Y,L)$ be the canonical section with $\operatorname{div}(s_F)=F$.  Choose a smooth
Hermitian metric $h_L$ on $L$ whose Chern form is $\theta_F$.   Set $\varphi_F:=\log\|s_F\|_{h_L}^2$, by the
Lelong--Poincar\'e equation,  \cite{Demailly-CAG} one has
$$
  [F]=\theta_F+\ddbar\varphi_F
$$
as closed $(1,1)$-currents.  Since $\theta_F$  is a closed real form, one has
$$
  [F]\wedge\theta_F^{j-1}-\theta_F^j
  =
  \ddbar\bigl(\varphi_F\theta_F^{j-1}\bigr)
$$
Let $\eta$ be a  smooth closed
$(n-j,n-j)$-form on $X$,  
 then $$
\int\ddbar\bigl(\varphi_F\theta_F^{j-1}\bigr)\wedge \pi^*\eta=0.$$
  Equivalently
\begin{equation}\label{BBB}\int_Y\theta_F^j\wedge\pi^*\eta
  =
  \int_F\theta_F^{j-1}\wedge\pi^*\eta.
\end{equation}
For the proper map $\pi$, the push-forward of a current $S$ is defined by duality:
$$
  \bigl\langle \pi_*S,\eta\bigr\rangle
  =
  \bigl\langle S,\pi^*\eta\bigr\rangle .
$$
The closed push-forward current  
$\pi_*([F]\wedge\theta_F^{j-1})$ has bidimension $(n-j,n-j)$ and is
supported on $Z$.   
If $j<c$, then
$$
  \dim Z=n-c<n-j.
$$
The first support theorem for closed currents on analytic sets
\cite[Theorem~III.2.10]{Demailly-CAG} implies that
$$
  \pi_*([F]\wedge\theta_F^{j-1})=0.
$$
By \eqref{BBB}, one has
$$\int_Y\theta_F^j\wedge\pi^*\eta=0.$$


For $j=c$, set
$$
  T:=\pi_*([F]\wedge\theta_F^{c-1}).
$$
Then $T$ is a closed current of bidimension $(m,m)$
supported on $Z$.  The second support theorem \cite[Theorem~III.2.13]{Demailly-CAG} implies that  
$$
  T=\lambda[Z]
$$
for some constant $\lambda$.  It suffices to compute $\lambda$ on  $Z_{\rm reg}$.   Put $U=X\setminus (Z\setminus Z_{\rm reg})$, so $U\cap Z=Z_{\rm reg}$.
Over $U$, the restriction of $F$ is reduced
along the unique exceptional divisor lying above $U\cap Z$, which we denote by $F_0$.  Then
$$
  q:=\pi|_{F_0}:F_0\longrightarrow Z_{\rm reg}
$$
is the projective bundle
$$
  F_0\simeq \mathbb P(N_{Z_{\rm reg}/X}).
$$
Here $\mathbb P(N_{Z_{\rm reg}/X})$ parametrizes lines in the normal bundle and
\begin{equation}\label{AAA}\mathcal O_Y(F)|_{F_0}=\mathcal O_{\mathbb P(N_{Z_{\rm reg}/X})}(-1). \end{equation}
All other components of $F$, except $F_0$, map into $Z\setminus Z_{\rm reg}$.
Thus, when restricting to the open set $U$,
$$
  T=\pi_*([F_0]\wedge\theta_F^{c-1}).
$$
Let $\gamma$ be a smooth $(m,m)$-form compactly supported in $U$, then
$$
  \langle T,\gamma\rangle
  =
  \int_{F_0}\theta_F^{c-1}\wedge q^*(\gamma|_{Z_{\rm reg}}).
$$
By \eqref{AAA} and the functoriality of first Chern class, the restriction $\theta_F|_{F_0}$ represents
\begin{equation}\label{CCC}c_1(\mathcal O_Y(F)|_{F_0})
  =
  c_1(\mathcal O_{\mathbb P(N_{Z_{\rm reg}/X})}(-1)).\end{equation}
By the projection formula for integration along the fibres of $q$,
$$
  \int_{F_0}\theta_F^{c-1}\wedge q^*(\gamma|_{Z_{\rm reg}})
  =
  \int_{Z_{\rm reg}}q_*(\theta_F|_{F_0}^{c-1})\,\gamma .
$$
The above integrals are well defined since $\gamma$ is compactly supported. 
 Then it follows from 
$$
  \int_{\mathbb P^{c-1}}(c_1(\mathcal O(-1))^{c-1}=(-1)^{c-1},
$$
and \eqref{CCC} that 
$$q_*(\theta_F|_{F_0}^{c-1})=(-1)^{c-1}.$$
Thus
$$
  \int_{F_0}\theta_F^{c-1}\wedge q^*(\gamma|_{Z_{\rm reg}})
  =
  (-1)^{c-1}\int_{Z_{\rm reg}}\gamma.
$$
Thus $T|_{Z_{\rm reg}}=(-1)^{c-1}[Z_{\rm reg}]$ and
$\lambda=(-1)^{c-1}$.  

Combining this with identity \eqref{BBB}, for every smooth closed
$(m,m)$-form $\eta$ on $X$ one has
$$
  \int_Y\theta_F^c\wedge\pi^*\eta
  =
  \langle T,\eta\rangle
  =
  (-1)^{c-1}\int_Z\eta .
$$
This completes the proof.
\end{proof}

\begin{lemma}
\label{lem:small-blowup-class}
In the setting of Lemma~\ref{lem:des}, assume $X$ is
compact K\"ahler and $\alpha$ is a K\"ahler class.  Then for all sufficiently small $t>0$,
$$
  L_t=\pi^*\alpha-tF
$$
is nef and big on $Y$. Moreover, $$
  L_t^n=\alpha^n+O(t^c)>0.
$$
\end{lemma}

\begin{proof}
The line bundle $\mathcal O_{\widehat X}(-E)$ is $p$-ample by the
construction of the blow-up.  Hence by \cite{Varouchas}, if $\omega\in\alpha$ is
K\"ahler, then
$$
  [p^*\omega]-tc_1(E)
$$
is a K\"ahler class in the sense of Grauert on the K\"ahler space $\widehat X$  for all sufficiently
small $t>0$.  Pulling this class back by the resolution
$\rho:Y\to\widehat X$ gives the nef class
$\pi^*\alpha-tF$ on $Y$.  By continuity 
$$
  L_t^n=\alpha^n+O(t^c)>0
$$
for $t>0$ sufficiently small.  Since $Y$ is
compact K\"ahler and $L_t$ is nef with positive top self-intersection, $L_t$
is big by the Demailly--Paun criterion \cite{DP}.
\end{proof}

\begin{proof}[\textbf{Proof of item (2) of Theorem~\ref{thm:semistable}}]
Since $(\alpha,\beta)$ is not semi-stable, there exists an irreducible analytic
subvariety $Z\subset X$ of dimension $m$, $1\le m<n$, such that
\begin{equation}\label{des}
  m\,\alpha^{m-1}\cdot \beta\cdot Z
  >
  \mu\,\alpha^m\cdot Z.
\end{equation}
Let $c=n-m$. If $c=1$, then $Z$ is an effective Cartier divisor. For all sufficiently small  $t>0$,
$$
  L_t=\alpha-tZ
$$
is a K\"ahler class.  Taking $\pi=\operatorname{id}_X$ and the
divisor $D=tZ$, one has
$$
  L_t^n
  =
  \alpha^n
  -
  n t\,\alpha^{n-1}\cdot Z
  +O(t^2),
$$
and
$$
  L_t^{n-1}\cdot \beta
  =
  \alpha^{n-1}\cdot \beta
  -
  (n-1)t\,\alpha^{n-2}\cdot \beta\cdot Z
  +O(t^2).
$$
Therefore
$$
  n\frac{L_t^{n-1}\cdot \beta}{L_t^n}
  =
  \mu+
  \frac{n}{\alpha^n}
  \left[
    \frac{\mu}{n}\,n\alpha^{n-1}\cdot Z
    -
    (n-1)\alpha^{n-2}\cdot \beta\cdot Z
  \right]t
  +O(t^2).
$$
The coefficient of $t$ is negative by the destablizing inequality \eqref{des} (Note here $m=n-1$).  Hence the slope of
$L_t$ is strictly smaller than $\mu$ for all sufficiently small 
$t>0$.

If $c\ge2$.  Let
$$
  p:\widehat X=\Bl_{\mathcal I_Z}X\to X
$$
be the blow-up of the reduced ideal of $Z$, let $E$ be the
total exceptional Cartier divisor defined by
$\mathcal I_Z\mathcal O_{\widehat X}=\mathcal O_{\widehat X}(-E)$, choose
a log resolution $\rho:Y\to\widehat X$ of $(\widehat X,E)$ as in
Lemma~\ref{lem:des}, and set
$$
  \pi=p\circ\rho,\qquad F=\rho^*E.
$$
By Lemma~\ref{lem:small-blowup-class},
$$
  L_t=\pi^*\alpha-tF
$$
is big and nef for all sufficiently small $t>0$.  Choose such a small
$t>0$; then $(\pi,tF)$ is an admissible test in the definition of minimal slope $\zeta_{\min}(\alpha,\beta)$.

By Lemma~\ref{lem:des},
$$
  L_t^n
  =
  \alpha^n
  -
  \binom{n}{c}t^c\,\alpha^m\cdot Z
  +O(t^{c+1}),
$$
and
$$
  L_t^{n-1}\cdot\pi^*\beta
  =
  \alpha^{n-1}\cdot \beta
  -
  \binom{n-1}{c}t^c\,\alpha^{m-1}\cdot \beta\cdot Z
  +O(t^{c+1}).
$$
Set
$$
  a=\binom{n}{c}\alpha^m\cdot Z,\qquad
  b=\binom{n-1}{c}\alpha^{m-1}\cdot \beta\cdot Z.
$$
Since $\alpha^{n-1}\cdot \beta/\alpha^n=\mu/n$,
$$
  n\frac{L_t^{n-1}\cdot\pi^*\beta}{L_t^n}
  =
  \mu+
  \frac{n}{\alpha^n}
  \left(
    \frac{\mu}{n}a-b
  \right)t^c
  +O(t^{c+1}).
$$
The coefficient of $t^c$ is negative exactly when $b>\frac{\mu}{n}a$, which is further equivalent to
$$
  m\,\alpha^{m-1}\cdot \beta\cdot Z
  >
  \mu\,\alpha^m\cdot Z.
$$
The above inequality is exactly the destabilizing inequality \eqref{des}, hence for all sufficiently small
 $t>0$,
$$
  n\frac{L_t^{n-1}\cdot\pi^*\beta}{L_t^n}<\mu.
$$
In particular $\zeta_{\min}(\alpha,\beta)<\mu$.
\end{proof}

\section{Analyticity of the set of optimally destablizing subvarieties}
In this section, we prove that the set of optimally destablizing  subvarieties $\NullJ(\alpha,\beta)$ associated to a semi-stable pair $(\alpha,\beta)$ is finite union of torus invariant subvarieties when $X$ is a smooth toric manifold, which finishes the proof of Theorem~\ref{thm:toric}. We remark that such a kind of result was proved in the two-dimensional complex by \cite{SD1} (see also \cite{FLSW}).

\begin{lemma}
Let $X$ be a smooth compact toric K\"ahler manifold.  If
$(\alpha,\beta)$ is $J$-slope semi-stable, then
$\NullJ(\alpha,\beta)$ is analytic.  More precisely, for each
$m=1,\ldots,n-1$, the equality locus
$$
  Y_m:=
  \bigcup_{\substack{Z\subset X\ \mathrm{irreducible}\\
          \dim Z=m\\
          \Theta_m\cdot Z=0}} Z
$$
is a finite union of torus orbit closures.
\end{lemma}

\begin{proof}
The torus $T\simeq(\mathbb C^*)^n$ is connected, so every
element $\lambda\in T$ is isotopic to the identity. Thus 
$\lambda$ acts trivially
on cohomology
$$
  \lambda^*\alpha=\alpha,\qquad \lambda^*\beta=\beta .
$$
Thus the class $\Theta_m:=\mu\alpha^m-m\alpha^{m-1}\beta$ is also fixed by $\lambda$.  Hence, if
$Z\subset X$ is an irreducible $m$-dimensional subvariety satisfying
$$
  \Theta_m\cdot Z=0,
$$
then its translate $\lambda\cdot Z$ also satisfies
$$
  \Theta_m\cdot (\lambda\cdot Z)=0 .
$$
Hence $Y_m$ is $T$-invariant.

We now prove that $Y_m$ is saturated by torus orbit closures. Fix an irreducible $m$-dimensional equality
subvariety $Z\subset X$ such that
$$
  \Theta_m\cdot Z=0 .
$$
Pick a point $x\in Z$ and let
$ y\in \overline{T\cdot x}$
 be a point in the complex analytic closure of $T\cdot x$. Choose a sequence $\lambda_j\in T$ such that
$\lambda_j\cdot x\to y.$

Note that the cycles $\lambda_j\cdot Z$ have the same homology class, this implies that  $\alpha$-volumes
$$
  \int_{\lambda_j\cdot Z}\alpha^m
$$
are constant and in particular uniformly bounded.  By Bishop-Lieberman's compactness for cycles on a compact
K\"ahler manifold \cite{Lieberman}, after passing to a subsequence, the cycles $\lambda_j\cdot Z$
converge to an effective $m$-cycle
$$
  Z_\infty=\sum_\ell a_\ell W_\ell,
 $$
where  $W_\ell$ are irreducible and reduced $m$-dimensional analytic
subvarieties and   $a_\ell>0$ are positive constants. It is clear that $y$ lies in the
support of $Z_\infty$.  Also, by continuity of integration over
cycles, one has
$$
  0
  =
  \Theta_m\cdot Z
  =
  \Theta_m\cdot Z_\infty
  =
  \sum_\ell a_\ell\,\Theta_m\cdot W_\ell .
$$
By $J$-slope semi-stability, for every $\ell$, one has
$$
  \Theta_m\cdot W_\ell\geq0,
$$
 which implies that
$$
  \Theta_m\cdot W_\ell=0.
$$
Thus every component $W_\ell$ of the limit cycle $Z_\infty$ is contained in $Y_m$.  Thus $y\in Y_m$ and we
 have proved that whenever $Y_m$ meets a torus orbit, it contains the closure of that
orbit.

Note that a compact toric manifold has only finitely many torus orbits, which implies that
$Y_m$ is a finite union of torus orbit closures.  Each orbit closure is
an analytic subvariety of $X$, so $Y_m$ is analytic.  Finally,
$$
  \NullJ(\alpha,\beta)
  =
  \bigcup_{m=1}^{n-1}Y_m
$$
is a finite union of analytic subsets.  This finishes the proof that
$ \NullJ(\alpha,\beta)$ is an analytic subset of $X$.
\end{proof}
\begin{proof}[\textbf{Proof of Theorem~\ref{thm:toric}}]
Next we address the high-order partial regularity of the weak solution $T$ in the large torus $(\mathbb{C}^*)^n$. For $0<\varepsilon<1$, perturb  the semi-stable pair $(\alpha,\beta)$ to the following stable pairs
$$
  \beta_\varepsilon=(1-\varepsilon)\beta+\varepsilon \alpha.
$$
Fix two K\"ahler forms $\omega,\chi$ in classes $\alpha,\beta$ separately.
By Collins-Sz\'ekelyhidi
\cite{CS}, for every K\"ahler form
$$\chi_\varepsilon=(1-\varepsilon)\chi+\varepsilon\omega\in\beta_\varepsilon,$$ there is a K\"ahler form
$\omega_\varepsilon=\omega+\ddbar\psi_\varepsilon\in\alpha$ such that
$$
  \chi_\varepsilon\wedge\omega_\varepsilon^{n-1}
  =
  \omega_\varepsilon^n,
  \qquad
  \sup{\psi_\varepsilon}=0.
$$

By the compactness of $\omega$-$\PSH$ functions, after passing to a subsequence, $\psi_\varepsilon\rightarrow \psi$ in the $L^1$ sense. When restricted to the large torus $(\mathbb{C}^*)^n$, $\psi_\varepsilon$ are smooth convex functions. Fix an open  convex $\Omega\subset\mathbb R^n$, then Rockafellar's Theorem \cite{Rockafellar} implies that $\psi_\varepsilon$ are uniformly bounded after shrinking $\Omega$. Then it follows from \cite[Proposition 27]{CS} that $\psi_\varepsilon$ has a uniform Hessian bound on any compact subset of $\Omega$. Finally, once we have uniform Hessian bound, the higher-order regularity follows by a standard argument.

\end{proof}


\end{document}